\documentclass[11pt]{amsart}
\usepackage{latexsym}
\usepackage{amsfonts,amssymb,amsmath,amsthm,amscd}
\usepackage[mathscr]{eucal}
\usepackage{graphicx}
\usepackage{subfigure}

\newcommand{\G}{\mbox{${\mathcal G}$}}

\newcommand{\E}{\mbox{${\mathcal E}$}}

\newcommand{\g}{\mbox{${\mathfrak g}$}}

\newcommand{\kf}{\mbox{${\mathfrak k}$}}

\newcommand{\p}{\mbox{${\mathfrak p}$}}

\newcommand{\C}{\mbox{${\mathbb C}$}}
\newcommand{\HH}{\mbox{${\mathbb H}$}}
\newcommand{\I}{\mbox{${\mathbb I}$}}

\newcommand{\PP}{\mbox{${\mathbb P}$}}

\newcommand{\R}{\mbox{${\mathbb R}$}}
\newcommand{\Z}{\mbox{${\mathbb Z}$}}

\newcommand{\tr}{{\rm tr}}
\newcommand{\ric}{{\rm Ric}}

\makeatletter
\def\numberwithin#1#2{\@ifundefined{c@#1}{\@nocnterrr}{%
  \@ifundefined{c@#2}{\@nocnterr}{%
  \@addtoreset{#1}{#2}%
  \toks@\expandafter\expandafter\expandafter{\csname the#1\endcsname}%
  \expandafter\xdef\csname the#1\endcsname
    {\expandafter\noexpand\csname the#2\endcsname
     .\the\toks@}}}}
\makeatother
\numberwithin{equation}{section}

\newtheorem{thm}[equation]{Theorem}
\newtheorem{lemma}[equation]{Lemma}
\newtheorem{prop}[equation]{Proposition}

\newtheorem{cor}[equation]{Corollary}
\newtheorem{ex}[equation]{Example}

\newtheorem{rem}[equation]{Remark}
\newenvironment{rmk}{\begin{rem} \em}{\end{rem}}

\begin{document}

\title{Non-K\"ahler Expanding Ricci Solitons II}
\author{M. Buzano}
\address{Department of Mathematics and Statistics, McMaster
University, Hamilton, Ontario, L8S 4K1, Canada}
\email{maria.buzano@gmail.com}
\author{A. S. Dancer}
\address{Jesus College, Oxford University, OX1 3DW, United Kingdom}
\email{dancer@maths.ox.ac.uk}
\author{M. Gallaugher}
\address{Department of Mathematics and Statistics, McMaster
University, Hamilton, Ontario, L8S 4K1, Canada}
\email{gallaump@mcmaster.ca}
\author{M. Wang}
\address{Department of Mathematics and Statistics, McMaster
University, Hamilton, Ontario, L8S 4K1, Canada}
\email{wang@mcmaster.ca}
\thanks{M. Wang is partially supported by NSERC Grant No. OPG0009421}

\date{revised \today}

\begin{abstract}
We produce new non-K\"ahler, non-Einstein, complete expanding gradient Ricci
solitons with conical asymptotics and underlying manifold of the form
$\R^2 \times M_2 \times \cdots \times M_r$, where $r \geq 2$ and $M_i$ are arbitrary
closed Einstein spaces with positive scalar curvature.
We also find numerical evidence for complete expanding solitons
on the vector bundles whose sphere bundles are
the twistor or ${\rm Sp}(1)$ bundles over quaternionic projective space.
\end{abstract}

\maketitle


\bigskip
\setcounter{section}{-1}

\section{\bf Introduction}

In \cite{BDGW} we constructed complete steady gradient Ricci soliton
structures (including Ricci-flat metrics) on manifolds of the form
${\R}^2 \times M_2 \times \ldots \times M_r$ where $M_i, \,2 \leq i \leq r,$
are arbitrary closed Einstein manifolds with positive scalar curvature. We also
produced numerical solutions of the steady gradient Ricci soliton equation on
certain non-trivial $\R^3$ and $\R^4$ bundles over quaternionic projective
spaces.  In the current paper we will present the analogous results for the case
of {\em expanding} solitons on the same underlying manifolds.

Recall that a gradient Ricci soliton is a manifold $M$ together with
a smooth Riemannian metric $g$ and a smooth function $u$, called the soliton
potential, which give a solution to
the equation:
\begin{equation} \label{gradRS}
{\ric}(g) + {\rm Hess} (u) + \frac{\epsilon}{2} \, g = 0
\end{equation}
for some constant $\epsilon$. The soliton is then called expanding, steady, or shrinking
according to whether $\epsilon$ is greater, equal, or less than zero.

A gradient Ricci soliton is called {\em complete} if the metric $g$ is complete. The
completeness of the vector field $\nabla u$ follows from that of the metric (cf \cite{Zh}).
If the metric of a gradient Ricci soliton is Einstein, then either ${\rm Hess}\, u = 0$
(i.e., $\nabla u$ is parallel) or we are in the case of the Gaussian soliton (cf \cite{PW}
or \cite{PRS}).

At present most examples of non-K\"ahlerian expanding solitons arise from left-invariant
metrics on nilpotent  and solvable Lie groups (resp. {\em nilsolitons}, {\em solvsolitons}),
as a result of work by J. Lauret \cite{La1}, \cite{La3}, M. Jablonski \cite{Ja}, and many
others (cf the survey \cite{La2}). These expanders are however not of gradient type,
i.e., they satisfy the more general equation
\begin{equation} \label{RS}
{\ric}(g) + \frac{1}{2} \,{\sf L}_{X} g + \frac{\epsilon}{2} \, g = 0
\end{equation}
where $X$ is a vector field on $M$ and ${\sf L}$ denotes Lie differentiation.

A large class of complete, non-Einstein, non-K\"ahlerian expanders of gradient type
(with dimension $\geq 3$)  consists of an $r$-parameter family
of solutions to (\ref{gradRS}) on $\R^{k+1} \times M_2 \times \ldots \times M_r$ where $k>1$
and $M_i$ are positive Einstein manifolds. The special case $r=1$ (i.e., no $M_i$) is due to
Bryant \cite{Bry} and the solitons have positive sectional curvature. The $r=2$ case
is due to Gastel and Kronz \cite{GK}, who adapted B\"ohm's construction of complete
Einstein metrics with negative scalar curvature to the soliton case. The case of
arbitrary $r$ was treated in \cite{DW3} via a generalization of the dynamical system
studied by Bryant. The soliton metrics in this family are all of multiple warped
product type. In other words, the manifold is thought of as being foliated by hypersurfaces
of the form $S^{k} \times M_2 \times \ldots \times M_r$ each equipped with a product metric
depending smoothly on a real parameter $t$.

More recently, Schulze and Simon \cite{SS} constructed expanding gradient Ricci solitons with
non-negative curvature operator in arbitrary dimensions by studying the scaling limits
of the Ricci flow on complete open Riemannian manifolds with non-negative bounded
curvature operator  and positive asymptotic volume ratio.

As pointed out in \cite{BDGW}, the situation of multiple warped products on
nonnegative Einstein manifolds is rather special because of the automatic lower
bound on the scalar curvature of the hypersurfaces. This leads, in the case where
all factors have positive scalar curvature, i.e., $k>1$, to definiteness of certain
energy functionals occurring in the analysis of the dynamical system arising from
(\ref{gradRS}), and hence to coercive estimates on the flow. In the present case,
where one factor is a circle, i.e., $k=1$, we can pass, as in \cite{BDGW}, to a subsystem
where coercivity holds, and this is enough for the analysis to proceed. The new solitons
obtained, like those of \cite{DW3}, have conical asymptotics, and are not of K\"ahler type
(Theorem \ref{mainthm1}). We note that the lowest dimensional solitons we obtain
form a $2$-parameter family on $\R^2 \times S^2$. The special case $r=1$ was
analysed earlier by the physicists Gutperle, Headrick, Minwalla and Schomerus \cite{GHMS}.

As in \cite{BDGW} we also obtain a family of solutions to our soliton equations
that yield complete Einstein metrics of negative scalar curvature (Theorem \ref{mainthm2}).
These are analogous to the metrics discovered by B\"ohm in \cite{Bo}. Recall
that for B\"ohm's construction the fact that the hyperbolic cone over the
product Einstein metric on the hypersurface acts as an attractor plays an
important role in the convergence proof for the Einstein trajectories. When $k=1$,
however, no product metric on the hypersurface can be Einstein with positive scalar
curvature, so the hyperbolic cone construction cannot be exploited directly. It
turns out that the analysis of the soliton case already contains most of the
analysis required for the Einstein case. The new Einstein metrics we obtain have
exponential volume growth.

Since the underlying smooth manifolds in the present paper are identical to those in
\cite{BDGW}, our constructions give rise to pairs of homeomorphic but not diffeomorphic
non-Einstein expanding gradient Ricci solitons as well as similar pairs of complete
Einstein manifolds with negative scalar curvature. Furthermore, since our expanders and
Einstein metrics have asymptotically conical structures, we also obtain
pairs whose asymptotic cones are homeomorphic but not diffeomorphic. The details
can be found at the end of \S 3.

To make further progress in the search for expanders, we need to consider more
complicated hypersurface types where the scalar curvature may not be bounded
below.  In \cite{BDGW} we carried out numerical investigations of steady solitons
where the hypersurfaces are the total spaces of Riemannian submersions for which
the hypersurface metric involves two functions, one scaling the base and one the fibre
of the submersion. We now look numerically at expanding solitons with such hypersurface
types, in particular where the hypersurfaces are $S^2$ or  $S^3$ bundles over
quaternionic projective space. We produce numerical evidence of complete expanding
gradient Ricci soliton structures in these cases.

Before undertaking our theoretical and numerical investigations, we first prove
some general results about expanding solitons of cohomogeneity one type. Some of the
results follow from properties of general expanding gradient Ricci solitons. However,
the proofs are much simpler and sometimes the statements are sharper, which is helpful
in numerical studies.  The results include monotonicity and concavity properties for
the soliton potential similar to those proved in \cite{BDGW} in the steady case,
as well as an upper bound for the mean curvature of the hypersurfaces. To derive this
bound, we need to know that complete {\it non-Einstein} expanding gradient Ricci solitons
have infinite volume. We include a proof of this fact here (Prop. \ref{logvol}) since
we were not able to find an explicit statement in the literature. Finally we derive an
asymptotic lower bound for the gradient of the soliton potential, which is in turn used
to exhibit a general Lyapunov function for the cohomogeneity one expander equations.

\section{\bf Background on cohomogeneity one expanding solitons}

We briefly review the formalism \cite{DW1} for Ricci solitons of cohomogeneity one.
We work on a manifold $M$ with an open dense set foliated by equidistant diffeomorphic
hypersurfaces $P_t$ of real dimension $n$. The dimension of $M$, the manifold where we
construct the soliton, is therefore $n+1$. The metric is then of the form
$\bar{g}=dt^2 + g_t$ where $g_t$ is a metric on $P_t$ and $t$ is the arclength
coordinate along a geodesic orthogonal to the hypersurfaces. This set-up is more general
than the cohomogeneity one ansatz, as it allows us to consider metrics with no symmetry
provided that appropriate additional conditions on $P_t$ are satisfied, see the following
as well as Remarks 2.18 and 3.18 in \cite{DW1}. We will also suppose that
$u$ is a function of $t$ only.

We let $r_t$ denote the Ricci endomorphism of $g_t$, defined by
${\rm Ric}(g_t)(X,Y) = g_t(r_t(X),Y)$ and viewed as an endomorphism via $g_t$.
Also let $L_t$ be the shape operator of the hypersurfaces, defined by the equation
$\dot{g_t} = 2 g_t L_t$ where $g_t$ is regarded as an endomorphism
with respect to a fixed background metric $Q$. The Levi-Civita
connections of $\bar{g}$ and $g_t$ will be denoted by
$\overline{\nabla}$ and $\nabla$ respectively. The relative volume $v(t)$
is defined by $d \mu_{g_t} = v(t)\, d \mu_Q$

We assume that the scalar curvature $S_t=\tr(r_t)$ and the
mean curvature $\tr(L_t)$ (with respect to the normal $\nu=\frac{\partial}{\partial t}$)
are constant on each hypersurface. These assumptions hold, for example,
if $M$ is of cohomogeneity one with respect to an isometric Lie group action.
They are satisfied also when $M$ is a multiple warped product over an
interval.

The gradient Ricci soliton equation now becomes the system
\begin{eqnarray}
  -{\rm tr} (\dot{L}) - {\rm tr} (L^2) + \ddot{u} + \frac{\epsilon}{2} &=& 0,  \label{TT} \\
    r - ({\rm tr}\, L)L - \dot{L} + \dot{u} L + \frac{\epsilon}{2} \, \I &=& 0, \label{SS}\\
d ({\rm tr} L) + \delta^{\nabla} L &=& 0.  \label{TS}
\end{eqnarray}
The first two equations represent the components of the equation
in the directions normal and tangent to the hypersurfaces $P$, respectively.
The third equation represents
the equation in mixed directions---here $\delta^{\nabla} L$ denotes the
codifferential for $TP$-valued $1$-forms.

In the warped product case the final equation involving the codifferential
automatically holds. This is also true for cohomogeneity one metrics that are
{\em monotypic}, i.e., when there are no repeated real irreducible summands in
the isotropy representation of the principal orbits (cf \cite{BB}, Prop. 3.18).

There is a conservation law
\begin{equation} \label{cons1}
\ddot{u} + (-\dot{u} + {\rm tr}\, L)\, \dot{u}  -\epsilon u = C
\end{equation}
for some constant $C$.
Using our equations we may rewrite this as
\begin{equation}  \label{cons2}
S + {\rm tr} (L^2) - (\dot{u} - {\rm tr}\, L )^2 - \epsilon u +
\frac{1}{2}(n-1) \epsilon = C.
\end{equation}
where  $S := {\rm tr} (r_t)$ is the scalar curvature $S$ of the principal
orbits. If $\bar{R}$ denotes the scalar curvature of the
ambient metric $\bar{g}$, then
$$ \bar{R} = -2 {\rm tr} (\dot{L}) - {\rm tr} (L^2) - ({\rm tr} L)^2 +S. $$
We can deduce the equality
\begin{equation}\label{ham}
\bar{R} + \dot{u}^2 + \epsilon u = -C -\frac{\epsilon}{2}(n+1).
\end{equation}

\medskip
We let $\xi$ denote the {\em dilaton mean curvature}
\[
 \xi := -\dot{u} + {\rm tr} \, L.
\]
 This is the mean curvature
of the dilaton volume element $e^{-u} d \mu_{\bar{g}}$.
It is often useful to define a new independent variable $s$ by
\begin{equation} \label{st}
\frac{d}{ds} := \frac{1}{\xi} \frac{d}{dt},
\end{equation}
and use a prime to denote $\frac{d}{ds}$. We note that equation
(\ref{TT}) implies that $\dot{\xi} = -{\rm tr} (L^2) + \frac{\epsilon}{2}$.

It is also useful, following \cite{DHW}, to introduce the quantity
\[
\E := C + \epsilon u.
\]
The conservation law may now be rewritten (for nonzero $\epsilon$) as
\begin{equation} \label{Econs}
\ddot{\E} + \xi \dot{\E} - \epsilon \E =0.
\end{equation}
Note that for a function $t \mapsto f(t)$, the
quantity $\ddot{f} + \xi \dot{f}$ is just the $u$-Laplacian in the
sense of metric measure spaces.

Another useful quantity is the normalised mean curvature
\[
{\mathcal H} = \frac{{\rm tr} L}{\xi} = 1 + \frac{\dot{u}}{\xi}
= 1 +u^\prime,
\]
which was introduced in \cite{DW3} and \cite{DHW}.

\medskip

We now specialise to the case of {\em expanding solitons}, that is
\[
\epsilon >0.
\]
We shall consider complete noncompact expanding solitons with one special
orbit. We may take the interval $I$ over which $t$ ranges to be
$[0, \infty)$ with the special orbit placed at $t=0$. Let $k$ denote the dimension
of the collapsing sphere at $t=0$. We will moreover assume in this section that
$u(0)=0$, since adding a constant to the soliton potential does not affect the equations.

A basic result of B.L. Chen \cite{Chb} together with the strong maximum principle
says that for a non-Einstein expanding gradient Ricci soliton $\bar{R} > -\frac{\epsilon}{2}(n+1)$.
So we deduce from (\ref{ham}) that
\[
\E < 0  \,\,\,\, \mbox{and} \,\,\,\, (\dot{u})^2 < -\E := -(C + \epsilon u).
\]
Using the first inequality and the smoothness conditions at $t=0$ we
find as in the steady case that
$\ddot{u}(0) = \frac{C}{k+1} <0$, so completeness imposes
restrictions on our initial conditions.

Integrating the second inequality and using the initial conditions yield
\begin{equation} \label{u-ineq}
0 \leq  -u(t) < \frac{\epsilon}{4} t^2 + \sqrt{-C} t
\end{equation}
and
\begin{equation} \label{udot-ineq}
|\dot{u}| < \frac{\epsilon}{2} t + \sqrt{-C}.
\end{equation}
These are just the cohomogeneity one versions of general estimates of
the potential due to Z.-H. Zhang \cite{Zh}.

\begin{prop} \label{expander-pot}
For a non-Einstein, complete, expanding gradient Ricci soliton
of cohomogeneity one with a special orbit, the soliton potential $u$
is strictly decreasing and strictly concave
on $(0, \infty)$.
\end{prop}

\begin{proof}
The conservation law (\ref{Econs})
and the fact that $\E$ is negative and $\epsilon$
is positive show that $u$ is strictly concave on a neighbourhood of
each critical point $t_0$. As we noted above, we also
have concavity at the special orbit $t=0$.
Now, as in the steady case \cite{BDGW}, we see
there are no critical points
of $u$ in $(0, \infty)$. As $\dot{u}(0)=0$,
we see $u$ is strictly decreasing on $(0, \infty)$.

Now set $y = \dot{u}$ and differentiate (\ref{cons1}); using (\ref{TT})
we obtain
\[
\ddot{y} + \xi \dot{y} - \left(\frac{\epsilon}{2} + {\rm tr}(L^2)\right) y =0.
\]
In particular $\ddot{y} + \xi \dot{y} <0$ since $y$ is negative.
Integrating shows $v e^{-u} \dot{y}$ is strictly decreasing, where
we recall that $v$ is the relative volume. As $t$ tends to $0$, the smoothness
conditions imply that $v e^{-u} \dot{y}$ tends to zero, so $\dot{y} = \ddot{u}$
is negative, as required.
\end{proof}

Our next result is inspired by the work of Munteanu-Sesum \cite{MS} for
the case of steady solitons.

\begin{prop} \label{vol}
For a non-Einstein, complete, expanding gradient Ricci soliton
of cohomogeneity one with a special orbit, the volume growth is
at least logarithmic.
\end{prop}

\begin{proof}
Let $M_t = \pi^{-1}([0,t])$ where $\pi$ is the projection of
$M$ onto the orbit space $I$. We consider the integral
\[
f(t) := \int_{M_t} \left( \bar{R} + \frac{\epsilon}{2}(n+1) \right) \;
d \mu_{\bar{g}}
\]
As we are considering non-Einstein solitons the integrand is positive.

Let $t_0 > 0$ and let $b := f(t_0)$.
Using the trace of the soliton equation and also the divergence theorem,
we have, for $t \geq t_0$:
\begin{eqnarray*}
0 < b  \leq f(t) &=& - \int_{M_t} \bar{\triangle} u \; d \mu_{\bar{g}} \\
                 &=& \int_{\partial M_t} (\bar{\nabla} u) \cdot(-\frac{\partial}
                      {\partial t}) \, d \mu_{\bar{g}}|_{\partial M_t} \\
                 &=& |\dot{u}|\, v(t)\\
                 &<&  (\frac{\epsilon}{2} t + \sqrt{-C} ) v(t)
\end{eqnarray*}
where we use (\ref{udot-ineq}) in the last line.
Hence $v(t) > \frac{b}{\frac{\epsilon}{2} t + \sqrt{-C}}$, and integrating
yields
\[
{\rm vol} (M_t) > {\rm vol} (M_{t_0}) -\frac{2b}{\epsilon}
\log (\frac{\epsilon}{2} t_0 + \sqrt{-C}) +\frac{2b}{\epsilon}
\log (\frac{\epsilon}{2} t + \sqrt{-C}).
\]
\end{proof}

\begin{prop} \label{MC}
Let $(M, \bar{g}, u)$ be
 a non-Einstein, complete, expanding gradient Ricci soliton
of cohomogeneity one with a special orbit.
Then there exists $t_1>0$ such that on $(t_1, \infty)$ we have
${\rm tr} \; L < \sqrt{\frac{n \epsilon}{2}}$.
\end{prop}

\begin{proof}
By Cauchy-Schwartz and the concavity result, we have
\begin{equation} \label{mcder}
\frac{d}{dt} ({\rm tr} \; L) < \frac{\epsilon}{2} - {\rm tr} \; (L^2)
\leq \frac{\epsilon}{2} -\frac{1}{n}( {\rm tr} \; L)^2.
\end{equation}
Note that by the smoothness conditions ${\rm tr} \; L$ is strictly
decreasing near $t=0$, and its limit as $t$ tends to zero from above is
$+ \infty$.

\medskip
(i) First let us assume that $\frac{d}{dt}({\rm tr} L)$ is
nonnegative at some $t_1$.
The above inequality shows that $| {\rm tr} L|^2 < \frac{\epsilon n}{2}$
at $t=t_1$.

Let us consider the solutions of the equation
\begin{equation} \label{heqn}
\dot{h} = \frac{\epsilon}{2} - \frac{1}{n} h^2.
\end{equation}
These are the family of increasing functions
\[
h(t) = \sqrt{\frac{\epsilon n}{2}} \left(
\frac{a \exp(t \sqrt{\frac{2 \epsilon }{n}}) -1}
{a \exp(t \sqrt{\frac{2 \epsilon }{n}}) +1} \right),
\]
where $a$ is a positive constant,
as well as the constant functions $\pm \sqrt{\frac{\epsilon n}{2}}$
which form the bounding envelope for this family. Hence
${\rm tr} \; L \leq  h^*(t) < \sqrt{\frac{\epsilon n}{2}}$ where
$h^*(t)$ is the solution to (\ref{heqn}) which agrees with ${\rm tr} \; L$
at $t_1$.

\medskip
(ii) Next suppose that $\frac{d}{dt} ({\rm tr} L)$ is always negative.
Now if ${\rm tr} L$ is ever zero then it is negative and bounded
away from zero on some semi-infinite interval. Recalling that
${\rm tr} L = \frac{\dot{v}}{v}$ and integrating, we see that the soliton volume
is finite, which contradicts Proposition \ref{vol}. So ${\rm tr} L$ is
positive on $(0, \infty)$, and, using Proposition \ref{expander-pot},
we see $\xi$ is also positive on this interval. Theorem 11 of \cite{PRS}
shows that $\xi$ tends to infinity as $t$ tends to $\infty$. But $\xi$ also
tends to infinity as $t$ tends to zero, so we have a minimum $t_1$
where $\dot{\xi}$ vanishes. Now (\ref{TT}) shows ${\rm tr} (L^2)=
\frac{\epsilon}{2}$ at $t_1$ and Cauchy-Schwartz shows $({\rm tr} L)^2 \leq
\frac{n \epsilon}{2}$ at $t_1$. As ${\rm tr} L$ is decreasing, we have the
desired result.
\end{proof}

\begin{rmk}
This bound on ${\rm tr} L$ is best possible, at least if we allow
the solitons to be Einstein. Indeed, the negative scalar curvature
Einstein metrics of B\"ohm \cite{Bo} give exactly this bound, as
${\rm tr} L$ is asymptotic to $\frac{n \epsilon}{2}$.
\end{rmk}

Next we consider properties of the Lyapunov function ${\mathscr F}_0$
which was introduced by  B\"ohm in \cite{Bo} for the Einstein case and
was subsequently studied in \cite{DHW} and \cite{BDGW} for the soliton case.
Note that this function was denoted by $\mathscr F$ in \cite{DHW}.

\begin{prop} \label{Lyapunov}
Let ${\mathscr F}_0$ denote the function $v^{\frac{2}{n}}\left( S + \tr( (L^{(0)})^2 ) \right)$
defined on the velocity phase space of the cohomogeneity one expanding gradient
Ricci soliton equations, with $L^{(0)}$ representing the trace-free part of $L$.
Then along the trajectory of a complete smooth non-Einstein expanding soliton, ${\mathscr F}_0$
is non-increasing for sufficiently large $t$.
\end{prop}

\begin{proof}
The formula for $\frac{d}{dt}{\mathscr F}_0$ in Proposition 2.17 of \cite{DHW} shows that
the proposition would follow if for sufficiently large $t$ one can show that
$$ \xi - \frac{1}{n} \tr \, L = -\dot{u} + \left(\frac{n-1}{n}\right) \,\tr \, L \geq 0.$$

We first note that $\tr \, L$ is eventually bounded below by $-\sqrt{\epsilon n/2}$.
Otherwise at some $t=t_1 > 0$, $\tr \, L \leq -\sqrt{\epsilon n/2}$ and (\ref{mcder}) shows
that this inequality continues to hold from $t_1$ onwards. But this would imply that
the soliton has finite volume, contradicting Proposition \ref{vol}.

We are now done since the next proposition (part (i)) shows that $|\dot{u}| = -\dot{u}$ grows at
least linearly for sufficiently large $t$. In particular, for large enough $t$, ${\mathcal F}_0$
fails to be strictly decreasing iff the shape operator of the hypersurfaces become diagonal.
\end{proof}

\begin{prop} \label{expander-gradlb}
Let $(M, \bar{g}, u)$ be a complete, non-Einstein, expanding gradient Ricci soliton
of cohomogeneity one with a special orbit. Suppose $t_1 > 2 \sqrt{\frac{5}{\epsilon}}$ and on $[t_1, +\infty)$ we have an
upper bound $\lambda_0 > 0$ for $\tr \, L$. Set $a:=\lambda_0 + \sqrt{-C}$. Then on $[t_1, +\infty)$
we have

\begin{itemize}
\item[$($i$)$]  $|\overline{\nabla} u| = -\dot{u}(t) >
      \frac{9}{10}\left(\frac{-\dot{u}(t_1)}{\frac{\epsilon}{2} t_1 + a} \right)
      \left(\frac{\epsilon}{2} t + a \right)$, \\

\item[$($ii$)$]   $\ddot{u} + \frac{\epsilon}{2} = -{\rm Ric}_{\bar{g}}(\frac{\partial}{\partial t},
         \frac{\partial}{\partial t}) \leq
         \frac{\epsilon}{2}\left( 1 +  \frac{9}{10} \frac{\dot{u}(t_1)}{\frac{\epsilon}{2} t_1 + a}   \right).$
\end{itemize}
\end{prop}

\begin{proof} By assumption and the upper bound (\ref{udot-ineq}) we have
$\xi < \frac{\epsilon}{2} t + a$.  Since $\dot{y} = \ddot{u} < 0$
and $y=\dot{u} < 0$ by Proposition \ref{expander-pot},  we see that $y$ satisfies the
differential inequality
$$ \ddot{y} + \left(\frac{\epsilon}{2}t + a \right)\, \dot{y} - \frac{\epsilon}{2} \,y < 0.$$

We will now compare $y$ with solutions of the corresponding equation
\begin{equation} \label{compeqn}
\ddot{x} + \left(\frac{\epsilon}{2}t + a \right)\, \dot{x} - \frac{\epsilon}{2} \,x = 0,
\end{equation}
which can be solved explicitly. This is because if we differentiate this equation, we obtain
$$ \frac{d^3 x}{dt^3} + \left(\frac{\epsilon}{2} t + a  \right)  \ddot{x} = 0,  $$
from which we can solve for $\ddot{x}$. Accordingly, upon integration
and using (\ref{compeqn}) we obtain
\begin{equation}  \label{soln-compeqn}
    x(t) = -\left(\frac{\epsilon}{2} t + a \right)\left( \frac{c_0}{\frac{\epsilon}{2} t_1 + a} -
     c_1 e^{\frac{\epsilon}{4}t_1^2 + at_1} \int_{t_1}^t \, \frac{e^{-\frac{\epsilon}{4}\tau^2 -a \tau} }{(\frac{\epsilon}{2} \tau +a)^2} \, d\tau   \right)
\end{equation}
where $c_0$ and $c_1$ are arbitrary constants.

In order to apply Theorem 13 on p. 26 of \cite{PrW}, we must choose $x(t_1) \geq y(t_1) = \dot{u}(t_1)$
and $\dot{x}(t_1) \geq \dot{y}(t_1) = \ddot{u}(t_1)$. Since $x(t_1) = -c_0$, we can maximize
$c_0$ by choosing $x(t_1) = \dot{u}(t_1)$. It follows that
$$ c_1 = -\ddot{x}(t_1) = -\frac{\epsilon}{2}\, x(t_1) +
     \left(\frac{\epsilon}{2} t_1 + a  \right)  \dot{x}(t_1) \geq -\frac{\epsilon}{2}\, \dot{u}(t_1) +
     \left(\frac{\epsilon}{2} t_1 + a  \right)  \ddot{u}(t_1).$$
In particular, an admissible choice for $c_1$ is $c_1 = \frac{\epsilon}{2}c_0 > 0$. With this choice, it remains
to find an upper bound for the integral in (\ref{soln-compeqn}).

To do this, we integrate by parts three times and then throw away the resulting term
involving integration (this term is negative). Specifically, we have
$$ \int_{\lambda_1}^{\lambda} \, \frac{e^{-\sigma^2/\epsilon}}{\sigma^2} \, d\sigma \leq
     \frac{\epsilon}{2}\left( \frac{e^{-\lambda_1^2/\epsilon}}{\lambda_1^3}\right)
     \left(1 - \frac{3}{2} \frac{\epsilon}{\lambda_1^2} + \frac{15}{\lambda_1^4} \left(\frac{\epsilon}{2} \right)^2
     - \left( \frac{\lambda_1}{\lambda}\right)^3 e^{-(\lambda^2 - \lambda_1^2)/\epsilon}
    \left(1 - \frac{3}{2} \frac{\epsilon}{\lambda^2} + \frac{15}{\lambda^4} \left(\frac{\epsilon}{2} \right)^2 \right)
     \right). $$
Using the change of independent variable $\lambda := \frac{\epsilon}{2} t +a$ and the fact that
$$ 1 - \frac{3\epsilon}{2} x + \frac{15}{4}\, \epsilon^2 x^2 =
     \left(1 - \frac{3\epsilon}{4}\, x \right)^2 + \frac{51}{16}\, \epsilon^2 x^2  \geq \frac{17}{20},$$
we obtain
\begin{eqnarray*}
    \lefteqn{ e^{\frac{\epsilon}{4}t_1^2 + at_1}
    \int_{t_1}^t \, \frac{e^{-\frac{\epsilon}{4}\tau^2-a \tau}}{(\frac{\epsilon}{2} \tau + a)^2}\, d\tau
         \leq  }  \\
  &    &   \hspace{1cm} \frac{1}{(\frac{\epsilon}{2} t_1 + a)^3}
               \left( 1 - \frac{\epsilon}{2} \frac{3}{(\frac{\epsilon}{2} t_1 +a)^2}
        + \left(\frac{\epsilon}{2} \right)^2 \frac{15}{(\frac{\epsilon}{2} t_1 + a)^4} -
        \frac{17}{20} \left(\frac{\frac{\epsilon}{2} t_1 + a}{\frac{\epsilon}{2} t + a} \right)^3
         \frac{e^{\frac{\epsilon}{4} t_1^2 + a t_1}  }{e^{ \frac{\epsilon}{4} t^2 + a t  } } \right).
\end{eqnarray*}

If we substitute the above information together with the choice $c_1 = \frac{\epsilon}{2}c_0$ in
the comparison inequality $\dot{u}(t) \leq x(t)$ (for $t \geq t_1$), we obtain
\begin{eqnarray*}
-\dot{u}(t) &  \geq & -\frac{\dot{u}(t_1)}{\frac{\epsilon}{2} t_1 + a} \left( \frac{\epsilon}{2} t + a \right)
           \left( 1 - \frac{\epsilon}{2} \frac{1}{ (\frac{\epsilon}{2} t_1 + a)^2 }
              \left( 1 - \frac{\epsilon}{2} \frac{3}{(\frac{\epsilon}{2} t_1 +a)^2}
        + \left(\frac{\epsilon}{2} \right)^2 \frac{15}{(\frac{\epsilon}{2} t_1 + a)^4} \right) \right)  \\
    & \geq &   -\frac{\dot{u}(t_1)}{\frac{\epsilon}{2} t_1 + a} \left( \frac{\epsilon}{2} t + a \right)
                 \left(  1 - \frac{\epsilon}{2} \frac{1}{ (\frac{\epsilon}{2} t_1 + a)^2 } \right) \\
      & > & \frac{9}{10} \left(  -\frac{\dot{u}(t_1)}{\frac{\epsilon}{2} t_1 + a}\right)
               \left( \frac{\epsilon}{2} t + a \right)
\end{eqnarray*}
where for the last inequality we used the hypothesis that $t_1 > 2 \sqrt{\frac{5}{\epsilon}}$, so that
$\frac{\epsilon}{2} t_1 + a > \sqrt{5 \epsilon}$. This completes the proof of (i).

The proof of (ii) follows by applying the same estimates to the comparison inequality
$\ddot{u}(t)=\dot{y}(t) \leq \dot{x}(t)$ for $t \geq t_1$. Note that by (2.2) in \cite{DW4}
and (\ref{SS}), the quantity $\ddot{u} + \frac{\epsilon}{2}$ is precisely the negative
of the Ricci curvature of the soliton metric in the direction $\frac{\partial}{\partial t}$.
\end{proof}

\begin{rmk}
In the above proof we can of course take $\lambda_0$ to be $\sqrt{\epsilon n/2}$ by
Proposition \ref{MC}. Notice, however, that in part (ii) of the proof of Proposition
\ref{MC} one automatically has an upper bound on $\tr \, L$. So one can apply Proposition
\ref{expander-gradlb} instead of Theorem 11 of \cite{PRS} to obtain a self-contained
proof for Proposition \ref{MC}.

Note also that both Propositions \ref{MC} and \ref{expander-gradlb} do not require
any curvature bounds.
\end{rmk}

We end this section with a simple generalization of Proposition \ref{vol} which, as far
as we know, has not been explicitly observed in the literature. An analogous result for
steady gradient Ricci solitons is Theorem 5.1 in \cite{MS}.

\begin{prop} \label{logvol}
A complete non-Einstein expanding gradient Ricci soliton has at least logarithmic
volume growth.
\end{prop}

\begin{proof} The basic idea is the same as that for the cohomogeneity one case.
Technically, we employ a formulation of the approximation arguments of Gaffney \cite{Gaf}
given by Yau in \cite{Y} (p. 660) which provides a compact exhaustion of the underlying
manifold with good properties for applying Stoke's theorem.

Let $(M, g, u)$ denote our non-Einstein expander, which is necessarily non-compact.
By going to the orientation double cover we may assume without loss of generality
that $M$ is orientable. Let us fix a point $p \in M$ and denote by $r(x)$ the distance
function from $p$, which is in general only Lipschitz continuous. Then for any value
$r>0$, there exists a smooth positive function $\varphi_r$ on $M$  such that
\begin{enumerate}
\item[(a)] except for finitely many $t < r$, $\varphi_r^{-1}(t)$ is a compact
    regular hypersurface in $M$
\item[(b)] $|d \varphi_r | \leq \frac{3}{2}$ on  $\varphi_r^{-1}([0,r])$
\item[(c)] for all $t \leq r$, $\varphi_r^{-1}(t) \subset B_p (t+1) \setminus B_p(t-1)$
\end{enumerate}
where $B_p(t)$ denotes the metric ball with centre $p$ and radius $t$.

We now consider the analogous function
$$ {\tilde f}(t) := \int_{B_p(t)} \, (R + \frac{\epsilon}{2} n)\, d\mu_g$$
where $R$ is the scalar curvature of $g$ and $n$ is the dimension of $M$. Note that
${\tilde f}$ is a non-decreasing function in $t$ and since the soliton is non-Einstein,
it follows from \cite{Chb} and the strong maximum principle that the integrand is strictly
positive, so that ${\tilde f}(t) > 0$ for $t>0$.

Let us choose $r \geq 3$ and $t$ to be between $2$ and $r$ such that $\varphi_r^{-1}(t)$ is a
closed regular hypersurface. If $M_t := \varphi_r^{-1}([0,t])$, then $M_t \subset B_p(r+1)$
and $\partial M_t = \varphi_r^{-1}(t)$. As in the proof of Prop \ref{vol} we have
$$ 0 < {\tilde b} := \int_{B_p(1)} \, (R + \frac{\epsilon}{2} n)\,  d\mu_g
     \leq \int_{M_t} (R + \frac{\epsilon}{2} n)\,  d\mu_g = -\int_{M_t} \Delta u \, d\mu_g,$$
where we have used the trace of the soliton equation. By Stoke's theorem, the last integral equals
$$    -\int_{\varphi_r^{-1}(t)} \, \nabla u \cdot \nu \, d\sigma_t $$
where $\nu$ denotes the unit outward normal along $\varphi_r^{-1}(t)$. It follows from
\cite{Zh} that the integrand can be bounded by ${\tilde c}\, (t+2)\, {\rm vol}_{n-1}(\varphi_r^{-1}(t))$
where $\tilde{c}$ is a positive constant which depends only on $n$ and $\epsilon$.
Therefore, except for a finite number of values of $t$, $2 \leq t \leq r$, we have
$$ \frac{{\tilde b}}{{\tilde c}(t+2)} \leq  {\rm vol}_{n-1}(\varphi_r^{-1}(t)).$$
Integrating this inequality from $2$ to $t$ and using the coarea formula together
with property (b) above, we obtain
$$ \int_2^t \,\frac{{\tilde b}}{{\tilde c}(\tau+2)}\, d\tau \leq \frac{3}{2}\, {\rm vol}_n (M_t)
        \leq \frac{3}{2}\, {\rm vol}_n (B_p(r+1)).$$
It follows that except for a finite number of $t$, $2 \leq t \leq r$,
we have ${\rm vol}_n (B_p(r+1)) \geq  \frac{2{\tilde b}}{3{\tilde c}} \log (1 + \frac{t}{2})$,
which yields for all $r$ ($r \geq 3$)
$${\rm vol}_n (B_p(r+1)) \geq  \frac{2{\tilde b}}{3{\tilde c}} \log (1 + \frac{r}{2}).$$
\end{proof}

\begin{rmk} \label{volume}
Of course there are non-compact negative Einstein manifolds with finite volume.
It is quite probable though that for non-trivial expanders the above volume lower bound is
not sharp. Most lower bounds for the volume in the literature involve additional assumptions
on the curvature. For example, in Proposition 5.1(b) of \cite{CaNi} or Theorem 1 of \cite{Chc}
a lower bound on the (average) scalar curvature is assumed.
\end{rmk}

\section{\bf Multiple warped product expanders}

In this section, we specialise to multiple warped products, that is
metrics of the form
\begin{equation} \label{metric}
{\bar{g}} = dt^2 + \sum_{i=1}^{r} g_i^2(t)\,  h_i
\end{equation}
on $I \times M_1 \times \cdots \times M_r$\, where $I$ is an interval in
$\mathbb R,$ $r \geq 2,$ and $(M_i, h_i)$ are Einstein manifolds
with real dimensions $d_i$ and Einstein constants $\lambda_i$. We observe
that
$n=\sum_i d_i$ is greater than or equal to $3$ as long as
 some $M_i$ is non-flat.

The Ricci endomorphism is now diagonal with components
given by blocks $\frac{\lambda_i}{g_i^2} \I_{d_i}$, where $i=1,\ldots, r$
and $\I_m$ denotes the identity matrix of size $m$.
We work with the variables
\begin{eqnarray}
X_i &=& \frac{\sqrt{d_i}}{\xi} \frac{\dot{g_i}}{g_i} \label{def-Xi}\\
Y_i &=& \frac{\sqrt{d_i}}{\xi} \frac{1}{ g_i} \label{def-Yi}\\
W   &=&  \frac{1}{\xi} := \frac{1}{-\dot{u} + {\rm tr} \; L} \label{def-W}
\end{eqnarray}
for $i=1, \ldots, r$.
The definition of $Y_i$ in \cite{DW2} and \cite{DW3} differs from that
above by a scale factor of $\sqrt{\lambda_i}$. This choice reflects the fact
that we are now allowing one of the $\lambda_i$ to be zero.
As in \cite{BDGW} we have
\[
\sum_{j=1}^{r} \, X_j^2  = \frac{{\rm tr} (L^2)}{\xi^2} \quad\textup{and}\quad
\sum_{j=1}^{r} \, \lambda_j Y_j^2 = \frac{ {\rm tr} (r_t)}{\xi^2}.
\]

As mentioned earlier, we shall
introduce the new independent variable $s$ defined by
(\ref{st}) and use  a prime ${ }^{\prime}$ to denote differentiation with
respect to $s$.

In these new variables the Ricci soliton system (\ref{TT})-(\ref{SS})
becomes
\begin{eqnarray}
X_{i}^{\prime} &=& X_i \left( \sum_{j=1}^{r} X_j^2 -1 \right) +
\frac{\lambda_i }{\sqrt{d_i}}\,\, Y_i^2 +
\frac{\epsilon}{2} (\sqrt{d_i} - X_i) W^2\, , \label{eqnX} \\
Y_{i}^{\prime} &=&  Y_i \left( \sum_{j=1}^{r} X_j^2 -\frac{X_i}{\sqrt{d_i}}
-\frac{\epsilon}{2} W^2 \right) \label{eqnY} \\
W^\prime &=& W \, \left( \sum_{j=1}^{r} X_j^2 - \frac{\epsilon}{2} W^2 \right) \label{eqnW}
\end{eqnarray}
for $i=1, \ldots, r$. Note that in the warped product situation, equation (\ref{TS})
is automatically satisfied.

As in \cite{BDGW} we use  ${\mathcal G}$ to denote $\sum_{i=1}^{r} X_i^2$.
The quantity ${\mathcal H} = W \, {\rm tr} \; L$ becomes
$\sum_{i=1}^{r} \sqrt{d_i} X_i$ in our new variables.
We further have the equation
\[
({\mathcal H} - 1)^\prime = ({\mathcal H} -1 )({\mathcal G}-1 -\frac{\epsilon}{2}
W^2 ) + \mathcal Q
\]
where
\begin{equation}  \label{modEng}
{\mathcal Q} = \sum_{i=1}^{r} \,(X_i^2 + \lambda_i Y_i^2) +
\frac{\epsilon(n-1)}{2} \, W^2 -1.
\end{equation}
As explained in \cite{DW3}, $\mathcal Q$ serves as an energy functional
in the expanding case, modifying the Lyapunov
\begin{equation} \label{Lyap}
{\mathcal L} :=  \sum_{i=1}^{r} \, (X_i^2 + \lambda_i Y_i^2)  - 1
\end{equation}
that plays a key role in the steady case (cf \cite{DW2},\cite{BDGW}).
The general conservation law (\ref{cons2}) then becomes
${\mathcal Q} = (C + \epsilon u)\, W^2.$

Note that in our situation, the quantity $\mathcal Q$ is no longer
a Lyapunov. However, we do have the equations
\begin{eqnarray*}
({\mathcal H} -1)^\prime &=& f_1 ({\mathcal H}-1) + f_2 {\mathcal Q} \\
{\mathcal Q}^\prime      &=& f_3 ({\mathcal H}-1) + f_4 {\mathcal Q}
\end{eqnarray*}
where $f_1 = G-1 -\frac{\epsilon}{2} W^2, f_2 = 1, f_3 = \epsilon W^2,$
and $f_4 = 2(G -\frac{\epsilon}{2} W^2)$. The crucial
point for us is that in the expanding case both $f_2$ and $f_3$
are positive, so the phase plane diagram in the
(${\mathcal H}-1, {\mathcal Q}$)-plane shows that the regions
$\{{\mathcal H} < 1, {\mathcal Q} < 0 \}$ and
$\{ {\mathcal H} > 1, {\mathcal Q} >0 \}$ are both flow-invariant.
Furthermore, the region $\{ {\mathcal Q} =0, {\mathcal H} =1 \}$ of
phase space corresponds to Einstein metrics of negative Einstein
constant and is of course also flow-invariant.

The above observations are in fact valid for the {\em general}
monotypic cohomogeneity one expanding soliton equations, not just
for the warped product case,
provided we make the general definition
\[
{\mathcal Q} := W^2 \E = W^2 (C + \epsilon u) \;\; \mbox{ and} \;\;
{\mathcal H} := W \,{\rm tr} \; L.
\]
(The conservation law shows that this is consistent with the
earlier formula for $\mathcal Q$ that we gave in the warped product case
(cf equation (4.6) in \cite{DHW})).
We refer to \cite{DHW} for a discussion of this topic
as well as the qualitatively different situation of shrinking solitons, where
$\epsilon$ is negative. However, apart from the multiple warped product
case, these formulae for ${\mathcal Q}$ involve polynomial or rational expressions
in the $X_i$ and $Y_i$ variables which need not be definite, so the estimates
obtained are not coercive.

In the warped product case with all $\lambda_i$ positive, which was
the situation examined in \cite{DW3}, $\mathcal Q$ is, as explained
above, a positive definite form (up to an additive constant) in the
$X_i,Y_i$, so we obtained coercive estimates which allowed us to
analyse the flow. For the rest of this section, we shall look at the
case where the collapsing factor $M_1$ is $S^1$, so $d_1=1$,
$\lambda_1=0,$ and the remaining Einstein constants $\lambda_i$ are
positive.  Then the equation for $X_1$ becomes
\[
X_{1}^{\prime} = X_1 \left( \sum_{j=1}^{r} X_j^2 -1 \right) +
\frac{\epsilon}{2} ( 1 - X_1) W^2.
\]
As $\mathcal Q$ now does not include a $Y_1$ term, the region
${\mathcal Q} < 0$ is no longer precompact. However, we will
see by using similar ideas as those in \cite{BDGW} that we can still
analyse the flow.

It is clear that we can recover $t$ and  $g_i$ from a solution
$X,Y,W$ of the above system via the relation $dt = W \; ds$ and the
formulae (\ref{def-Xi}), (\ref{def-Yi}), (\ref{def-W}).
As usual we choose $t=0$ to correspond to $s = -\infty$.
The soliton potential $u$ is recovered from integrating
\begin{equation} \label{def-u}
 \dot{u} = \tr(L) - \frac{1}{W} =
\frac{{\mathcal H} - 1}{W} = \frac{\sum_{i=1}^{r} \sqrt{d_i} X_i -1}{W}.
\end{equation}

We next compute the critical points of the soliton system (\ref{eqnX})-(\ref{eqnW}).

\begin{lemma} \label{zeros}
Let $d_1 =1$ and $d_i > 1$ for $i > 1$, so that
$\lambda_i =0$ iff $i=1$. The stationary points of $($\ref{eqnX}$)$, $($\ref{eqnY}$)$,
$($\ref{eqnW}$)$ in $X,Y,W$-space consist of
\begin{enumerate}
\item[$($i$)$] the origin

\item[$($ii$)$] points with $W=0$, $Y_i=0$ for all $i$, and $\sum_{i=1}^{r} X_i^2 =1$

\item[$($iii$)$] points given by
\[
W=0 \;\; : \;\; X_i = \sqrt{d_i}\, \rho_A \;\; : \;\; Y_i^2 = \frac{d_i}{\lambda_i}\,\, \rho_A (1- \rho_A),
\;\;\;  i \in A
\]
and $X_i = Y_i =0$ for $i \notin A$, where $A$ is any nonempty subset
of $\{2, \ldots, r \}$, \\ and $\rho_A = \left( \sum_{j \in A} d_j \right)^{-1}$

\item[$($iv$)$] the line where $W=0$, $X_i=0$ for all $i$, and $Y_i=0$ for $i >1$

\item[$($v$)$] the line where $W=0$, $X_1 =1$, and $X_i, Y_i =0$ for $i > 1$.

\item[$($vi$)$] the points $E_{\pm}$ with coordinates
\[
X_i = \frac{\sqrt{d_i}}{n} \;\; : \;\; Y_i =0 \;\; : \;\; W =
\pm \sqrt{\frac{2}{n \epsilon}}\, .
\]     \qed
\end{enumerate}
\end{lemma}

Note that $\mathcal L$ equals $-1$ in case (i) and (iv), equals $0$ in case
(ii), (iii) and (v), and equals $\frac{1-n}{n}$ in case(vi).
Also $\mathcal Q$ is $-1$ in cases (i) and (iv) and zero otherwise.
   Cases (i)-(v) arose
in \cite{BDGW} in the steady case. Case (vi) is special to the
expanding case and arose in \cite{DW3}. Again the origin is no longer
an isolated critical point.

The analysis of the equations is quite similar to that in \cite{DW3}, with
appropriate changes as in \cite{BDGW} to reflect the fact that one factor
$M_1$ of the product hypersurface is flat. Accordingly we shall be
brief in our discussion.

We look for solutions where the flat factor $M_1 = S^1$ collapses at the end
corresponding to $t=0$ (that is, $s = -\infty$).
In our new variables, this translates into considering trajectories in
the unstable manifold of the critical point $P$ of (\ref{eqnX})-(\ref{eqnW})
(of type (v)) given by
\[
W =0, \;\; X_1 = 1, \;\; Y_1 = 1, \;\; X_i =Y_i=0  \,\, (i > 1).
\]
Note that at this critical point we have
${\mathcal L}= {\mathcal Q} =0$ and $\G = {\mathcal H} =1$.

The linearisation about this critical point is the system
\begin{eqnarray*}
x_1^{\prime} &=& 2x_1 \\
y_1^\prime &=& x_1 \\
x_i^\prime &=& 0 \;\;\;\,\, (i \geq 2) \\
y_i^\prime &=& y_i \;\;\,\, (i \geq 2) \\
w^{\prime} &=& w
\end{eqnarray*}
with eigenvalues
$2$,  $1$ ($r$ times), and $0$ ($r$ times).

\medskip
The results of \cite{Buz} now show we have an $r$-parameter family
of trajectories $\gamma(s)$ emanating from $P$ and pointing into the region
$\{{\mathcal Q} < 0, {\mathcal H} < 1\}$. Moreover, by the arguments above,
such trajectories stay in this region. We can choose the trajectories to have
$W, Y_i$ positive for all time, as the loci $\{Y_i =0 \}$ or $\{W=0 \}$ are flow-invariant
and the equations are invariant under changing the sign of $W$ and/or of
any $Y_i$.

As mentioned above, as $M_1$ is flat and $Y_1$ does not appear in $\mathcal Q$,
the region $\{{\mathcal Q} < 0 \}$ is no longer precompact. However, since the variable
$Y_1$ only enters into the equations through the equation for $Y_1^\prime$,
we may follow \cite{BDGW} and consider the subsystem obtained by
omitting the $i=1$ equation in (\ref{eqnY}). The result is a system of equations
in $W$, $X_i \,(i=1, \ldots, r)$ and $Y_i \,(i=2, \ldots, r)$, and on this
$2r$-dimensional phase space the locus $\{{\mathcal Q} < 0\}$ {\em is} precompact.
Once we have a long-time solution  to the subsystem, $Y_1$ may be recovered via
\[
Y_{1}(s) = Y_{1}(s_0) \exp \left(\int_{s_0}^{s} \sum_{j=1}^{r} X_j^2 -X_1 - \frac{\epsilon}{2} W^2
\right)
\]
where $s_0$ is a fixed but arbitrary constant.

The critical points of the subsystem are obtained by removing
the $Y_1$-coordinate from those of the full system. In particular, the
origin becomes an isolated critical point, and case (v) of Lemma \ref{zeros}
gives rise to the special critical point $\hat{P}$ with
$W=0$, $X_1=1$, $X_i=0 \; (i>1)$, $Y_i=0 \; (i=2,\ldots,r)$,
from which emanates an $r$ parameter family of local solutions lying in the
region $\{W > 0, Y_i > 0 \, (i>1), {\mathcal Q} < 0, {\mathcal H} < 1 \}$.
The $r$ parameters may be thought of as $g_i(0), i>1$ and the constant $C$
in the conservation law (which has to be negative under the assumption that
$u(0)=0$). Homothetic solutions are eliminated by fixing the value of $\epsilon$.

Precompactness of the region where the subsystem flow lives shows that
the variables are bounded, so that the flow exists for all $s$.
Hence the same is true for the original flow also. As in Lemma 2.2 of
\cite{DW3} we can show that $X_i$ are positive for all $s$. It follows that
${\mathcal H} > 0$ and $X_i < \frac{1}{\sqrt{d_i}}$.
Furthermore, we still have the equation
\[
\left( \frac{W}{Y_i} \right)^{\prime} = \frac{X_i}{\sqrt{d_i}}
\left( \frac{W}{Y_i} \right),
\]
including the possibility $i=1$. So $\frac{W}{Y_i}$ increase monotonically
to limits $\sigma_i \in (0, \infty]$. (We shall presently show that $\sigma_i$
must all be equal to $+\infty$.)

As the trajectories of interest lie in a precompact set, each of them
has a nonempty $\omega$-limit set $\Omega$, where we suppressed the dependence
on the trajectory. Moreover, each $\Omega$ is compact, connected, and invariant
under both forward and backward flows.

As in \cite{DW3} (p. 1115) we can show that $\Omega$ lies in the
locus $ \{Y_i =0, 2 \leq i \leq r  \}$. Now on this locus the flow is just
the same as that in \cite{DW3}, and the arguments there (cf pp. 1116-1120)
show as before that $\Omega$ contains the origin (in the phase space for
the subsystem). The centre manifold argument on pp. 1121-1122 of \cite{DW3}
then shows the origin is a nonlinear sink, so in fact the trajectory converges
to the origin.

Now we can follow the arguments for Lemma 3.13 in \cite{DW3} to show that
\begin{equation} \label{asymplim}
\lim_{s \rightarrow \infty} \frac{X_i}{W^2} = \Lambda_i :=
\frac{\lambda_i}{\sigma_i^2 \sqrt{d_i}}  + \frac{\epsilon}{2} \sqrt{d_i},
\end{equation}
where $\Lambda_i >0$. This is valid in particular for $i=1$, in which case
$\Lambda_1 = \frac{\epsilon}{2}$. In fact, the proof of Lemma
3.15 in \cite{DW3} shows that $\sigma_i$ cannot be finite, and so
$\frac{\Lambda_i}{\sqrt{d_i}} = \frac{\epsilon}{2}$ for all $i$. Applying
this fact to the relation $ \frac{\dot{g}_i}{g_i} = \frac{1}{\sqrt{d_i}} \frac{X_i}{W} =
\frac{1}{\sqrt{d_i}} \frac{X_i}{W^2} W$, it follows that the
hypersurfaces have asymptotically decaying principal curvatures.

The limits (\ref{asymplim}) also imply that, for sufficiently large $s$, there exist
$a_1, a_2 > 0$ such that $a_1 W^4 \leq \G \leq a_2 W^4$, from which we deduce
completeness of the soliton metric by using the relation $dt = W ds$
and the equation (from (\ref{eqnW})) $W ds = \frac{dW}{\G - \frac{\epsilon}{2} W^2}$.
We further have $W \sim \frac{1}{\sqrt{\epsilon s}}$ and $s \sim
\frac{\epsilon t^2}{4}$.

The asymptotics for $g_i, i>1,$ are deduced as in \cite{DW3}. As for $g_1$, the equation
\[
\left( \frac{W}{Y_1} \right)^{\prime} = \frac{X_i}{\sqrt{d_i}}
\left( \frac{W}{Y_1} \right)
\]
and $X_1 \sim \frac{\epsilon}{2} W^2 \sim \frac{1}{2s}$
show that $g_1 = \frac{W}{Y_1}$ is also asymptotically linear in $t$, so
we have conical asymptotics for all factors.

\begin{rmk}
This contrasts with the steady case, where the asymptotic geometry for
$n=1, r=1$ (the cigar) is different from the paraboloid
asymptotics for the Bryant solitons with $n>1, r=1$. In the steady
case with $r>1$ our work in \cite{BDGW} yielded solitons of {\em mixed} asymptotic
 type, where $g_1$ tended to a positive constant and $g_i$ for $i>1$ behave
like $\sqrt{t}$.

In the expanding case, both the $n=1, r=1$ case (due to \cite{GHMS})
and the $n>1, r=1$ case (due to Bryant \cite{Bry}) have conical asymptotics,
and our solutions here for the $r>1$ case also exhibit conical behaviour.
\end{rmk}

We summarise the discussion in this section by the following

\begin{thm} \label{mainthm1}
Let $M_2, \ldots, M_r$ be closed Einstein manifolds with positive scalar
curvature. There is an $r$ parameter family of non-homothetic complete smooth
expanding gradient Ricci soliton structures on the trivial rank $2$ vector bundle
over $M_2 \times \ldots \times M_r$, with conical asymptotics in the sense given
above.     \qed
\end{thm}

\begin{rmk} \label{ricci}
As in \cite{DW3}, we can see directly from the equations that the soliton
potential $u$ is concave, in accordance with Proposition \ref{expander-pot}.
We can similarly deduce directly that $\ric(\bar{g}) + \frac{\epsilon}{2}\, \bar{g}$
is positive semidefinite, so $-u$ is subharmonic.

Next we note that when $r \geq 2$, the sectional curvatures $\kappa(X \wedge Y) $,
for $X, Y$  tangent to different Einstein factors, satisfy
$ -\frac{c_1}{t^2} \leq \kappa(X \wedge Y) \leq -\frac{c_2}{t^2} < 0$ for certain
positive constants $c_1, c_2$. This shows that the hypothesis of
$\lim_{t\rightarrow \infty} t^2|{\rm sect}| = 0$ in many results in \cite{Chc} is not
satisfied by our examples. In particular, the simplest hypersurface type in our
examples is $S^1 \times S^{n-1}$ (cf Theorem 4 in \cite{Chc}).

Furthermore, all sectional curvatures decay faster than $t^{-2 + \delta}$ for an
arbitarily small $\delta > 0$. Hence the ambient scalar curvature $\bar{R}$ tends
to zero. Finally we note that none of the hypotheses (topological or metric) in the recent
rigidity theorem of Chodosh \cite{Cho} are satisfied by our examples.
\end{rmk}

\section{\bf Complete Einstein metrics with negative scalar curvature}

We may also consider the flow of equations (\ref{eqnX})-(\ref{eqnW})
in the variety $\{{\mathcal Q}=0, {\mathcal H}=1 \}$. Such solutions
of course correspond to Einstein metrics with negative scalar curvature,
the soliton potential now being constant. In the case when $d_i>1$
for all $i$, such metrics were constructed earlier in \cite{Bo}
by dynamical systems methods as well. In \cite{DW3} we pointed out (in Remark
4.13 there) that a simpler proof of B\"ohm's result can be obtained using our special
variables and the embedding of the Einstein system within the soliton system.

In the present situation, where $d_1=1$, the hypersurfaces in
the multiple warped product no longer admit a positive Einstein
product metric whose hyperbolic cone acts as an attractor for the Einstein
system. Nevertheless our setup allows us easily to deduce the following

\begin{thm} \label{mainthm2}
Let $M_2, \ldots, M_r$ be compact Einstein manifolds with positive scalar
curvature. There is an $r-1$ parameter family of non-homothetic complete smooth
Einstein metrics on the trivial rank $2$ vector bundle over $M_2 \times \ldots \times M_r$.
\end{thm}

To prove the theorem, we consider the $r-1$ parameter family of trajectories
emanating from the critical point $P$ and lying in the variety
$\{{\mathcal Q}=0$, ${\mathcal H}=1\}$. Note that this variety is smooth.

As in the previous section, we see that the flow is defined for all $s$
by first restricting to the subsystem obtained by omitting the equation
for $Y_1$ and observing that the locus $\{{\mathcal Q}=0\}$ is compact.
As usual we can take $Y_i, W$ positive on our trajectories, and we can show
$X_i$ are positive also. In the following we will work with the subsystem.

The $\omega$-limit set $\Omega$ of a fixed trajectory will lie within
the locus $\{Y_i =0 : i=2,\ldots, r \}$ by the same argument as in the soliton
case. However, the difference now is that no point in $\Omega$ can have
$W$-coordinate equal to $0$. Otherwise, ${\mathcal G}= 1$ and such a point
is a critical point of type (ii) in Lemma \ref{zeros}. The argument in the last
part of the proof of Proposition 3.6 in \cite{DW3} then leads to a contradiction.
This in particular implies that the only critical point of the flow lying in $\Omega$
is $E_{+}$ (since $W>0$ along our trajectory).

We next consider the trajectory starting from a non-critical point in $\Omega$.

Recall from \cite{DW3} that on the locus
$\{{\mathcal Q}=0, {\mathcal H}=1, Y=0 \}$, the quantity
$J := G -\frac{\epsilon}{2} W^2$ satisfies $0 \leq J \leq 1$
and the equation
\[
J^{\prime} = 2J (J-1).
\]
Moreover, $J=1$ exactly when $W=0$ and $\G = 1$, and $J=0$ exactly
at the critical points $E_{\pm}$ (of type (vi) in Lemma \ref{zeros}).
Points with $W>0$ (resp. $W<0$) flow to $E_{+}$ (resp. $E_{-}$) and
flow backwards to $W=0$.

For our trajectory $W$ is necessarily positive, so we obtain a
contradiction since $\Omega$ is compact, flow-invariant, and contains no
point with zero $W$-coordinate. We therefore deduce that $\Omega$ is $\{E_+ \}$.
Now it was observed in Lemma 3.8 of \cite{DW3} that for the
flow on $\{{\mathcal Q} =0, {\mathcal H}=1\}$, the point $E_{+}$
is a sink, so our (original) trajectory converges to $E_{+}$.

As $dt = W ds$ and $W$ is converging to a positive constant we deduce
the metric is complete. Using (\ref{def-Xi}) we see that the metric
components $g_i^2$ grow exponentially fast asymptotically.

\medskip

We end this section with some consequences of combining our existence theorems
with  a study of the differential topology of some of our examples.

\medskip

We will focus on the case where $r=2$ and $M_2$ is a homotopy sphere.
Recall that Boyer, Galicki and Koll\'ar \cite{BGK1}, \cite{BGK2} have constructed
Sasakian Einstein metrics with positive scalar curvature on all Kervaire spheres
(with dimension $4m+1$) and those homotopy spheres of dimension $7, 11,$ or $15$
which bound parallelizable manifolds. As in \cite{BDGW} we can take these Einstein
manifolds or the standard sphere as $M_2$ in our constructions in \S 2 and \S 3.
Since it follows from the independent work of K. Kawakubo \cite{Ka} and
R. Schultz \cite{Sc} that the manifolds $\R^2 \times M_2$ and $\R^2 \times S^q$
are not diffeomorphic if $M_2$ is an exotic sphere (cf \cite{KwS}), we deduce the following

\begin{cor} \label{diffstr}
In dimensions $9, 13, 17$ and all dimensions $4m+3$ with $m \neq 0, 1, 3, 7, 15, 31$
there exist pairs of homeomorphic but not diffeomorphic manifolds both of which admit
non-Einstein, complete, expanding gradient Ricci soliton structures. The same holds
for complete Einstein metrics with negative scalar curvature.    \qed
\end{cor}

Note also that our expanding gradient Ricci solitons and negative Einstein manifolds
also exhibit conical asymptotics. The corresponding cones are differentially of the form
$\R_{+} \times S^1 \times M_2$, where $\R_{+}$ is the set of positive real numbers.
We are indebted to Ian Hambleton for providing an outline of the proof of the following
consequence of the above-mentioned work of Kawakubo and Schultz.

\begin{prop} \label{diff-str-cone}
Let $\Sigma^q$ and $S^q$ denote respectively a non-standard homotopy sphere and
the standard $q$-sphere. Then the open cones $\R_{+} \times S^1 \times \Sigma$
and $\R_{+} \times S^1 \times  S^q $ are not diffeomorphic.
\end{prop}

\begin{proof} (I. Hambleton)
Let $\phi: \R_{+} \times S^1 \times \Sigma^q \longrightarrow \R_{+} \times S^1 \times S^q$
be an orientation preserving diffeomorphism. For convenience, let $X = S^1 \times \Sigma^q$,
$Y = S^1 \times S^q$, and $X_a = \{a\} \times X$, $Y_b = \{b\} \times Y$. By compactness,
$\phi(X_1) \subset (a, b) \times Y$ for some $0< a < b$. Moreover, by Alexander duality
(applied e.g. to $(a, b) \times Y$ with the ends capped off by attaching $D_{\pm}^2 \times Y$),
$\phi(X_1)$ is a two-sided hypersurface that separates $(a, b) \times Y$ into two path-connected
open submanifolds of $\R_{+} \times Y$.

Let $W_{\pm}$ denote the closures of these path components. Then, using the diffeomorphism
$\phi$, which has to preserve the ends of $\R_{+} \times X$ and $\R_{+} \times Y$,
one easily sees that $W_{-}$ (resp. $W_{+}$) is a compact manifold whose boundary
consists of $Y_a$ and $\phi(X_1)$ (resp. $\phi(X_1)$ and $Y_b$). Moreover, by composition
with suitable retractions and the restrictions of $\phi$ or $ \phi^{-1}$ to suitable
subsets, one also sees easily that the inclusion of the boundary components into $W_{-}$
are homotopy equivalences, i.e., $W_{-}$ is an $h$-cobordism between its boundary components.
Noting that the Whitehead group of $\pi_1(S^1 \times S^q) = \Z$ is trivial and
applying the $s$-cobordism  theorem, we get a contradiction to the result of Kawakubo and
Schultz that $S^1 \times \Sigma^q$ and $S^1 \times S^q$ are not diffeomorphic.
\end{proof}

Hence we obtain for the dimensions given in Corollary \ref{diffstr} pairs of
non-Einstein complete expanding gradient Ricci solitons (or complete negative
Einstein manifolds) whose asymptotic cones are homeomorphic but not diffeomorphic.

\section{\bf Numerical examples}

We shall now look at some numerical solutions of the equations
(\ref{TT})-(\ref{TS}). The Ricci soliton equation in the cohomogeneity one
setting has an irregular singular point at $t=0$. We therefore follow the
procedure in \cite{DHW}, \S 5 and \cite{BDGW}. That is, we first find a series
solution in a neighbourhood of the singular orbit satisfying the appropriate
smoothness conditions. We then truncate the series and use the values of the
resulting functions at some small $t_0 > 0$ as initial values to generate solutions
of the equations for $t > t_0$ via a fourth order Runge-Kutta scheme. Because
the manifolds we are considering are non-compact, we check the numerics obtained
against the general asymptotic properties given in the first section.

The explicit cases that we shall look at are those where the hypersurface is
the twistor space of quaternionic projective space and the total space of the
corresponding ${\rm Sp}(1)$ bundle. For these examples, the estimates
${\mathcal Q}<0$ and ${\mathcal H} < 1$ do not give coercive estimates,
and we do not yet have analytical existence proofs. However the numerics
give a strong indication that complete expanding solitons exist
in these cases.

Let us recall the equations that will be analysed numerically, following
\cite{BDGW}. We consider cohomogeneity one manifolds with principal orbits
$G/K$ whose isotropy representation consists of two inequivalent
${\rm  Ad}(K)$-invariant irreducible real summands. We assume that $K
\subset H \subset G$ where $H, K$ are closed subgroups of the compact
Lie group $G$ such that $H/K$ is a sphere. A $G$-invariant
background metric $\mathsf b$ is chosen on $G/K$ such that it induces the
constant curvature $1$ metric on $H/K$. The cohomogeneity one manifolds
are then the vector bundles $G \times_{H} \R^{d_{1}+1}$ where $H/K
\subset \R^{d_{1}+1}$ is regarded as the unit sphere.

Let $\g = \kf \oplus \p$ be an ${\rm Ad}(K)$-invariant decomposition
of the Lie algebra of $G$, so that $\p$ is identified with the tangent
space of $G/K$ at the base point.  We can further decompose $\p$ into
irreducible $K$-modules; thus $\p = \p_1 \oplus \p_2$ where $\p_1$ and
$\p_2$ are respectively the tangent spaces (at the base point) to the sphere
$H/K$ and the singular orbit $G/H$. Their respective dimensions
are denoted by $d_1$ and $d_2$.

The metrics of cohomogeneity one take the form
$$ \bar{g} = dt^2 + g_1(t)^2 \; {\mathsf b}|\p_1 + g_2(t)^2 \; {\mathsf b}|\p_2.$$
Letting $(z_1, \ldots, z_6):=(g_1, \dot{g_1}, g_2, \dot{g_2}, u, \dot{u})$, the
gradient Ricci soliton equations become
\begin{eqnarray*} \label{2summands}
\dot{z_1} & = & z_2  \\
\dot{z_2} & = & -(d_1 -1)\frac{z_2^2}{z_1} -d_2 \frac{z_2 z_4}{z_3} + z_2 z_6 + \frac{d_1 -1}{z_1}
            + \frac{A_3}{d_1} \frac{z_1^3}{z_3^4} + \frac{\epsilon}{2}z_1 \\
\dot{z_3} & = & z_4 \\
\dot{z_4} & = & -d_1 \frac{z_2 z_4}{z_1} -(d_2-1)\frac{z_4^2}{z_3} + z_4 z_6
        + \frac{A_2}{d_2} \frac{1}{z_3} -2\frac{A_3}{d_2} \frac{z_1^2}{z_3^3} + \frac{\epsilon}{2}z_3 \\
\dot{z_5} & = & z_6 \\
\dot{z_6} & = & -z_6 \left( d_1 \frac{z_2}{z_1} + d_2 \frac{z_4}{z_3} \right)
            + z_6^2 + \epsilon z_5 + C,
\end{eqnarray*}
where $A_i$ are positive constants which appear in the scalar curvature function
of the principal orbit. Note that because of the backgound metric chosen, the
coefficient $\frac{A_1}{d_1}$ of the $\frac{1}{z_1}$ term in the second equation
becomes $d_1 -1$, and for expanding solitons we have $\epsilon>0$.

Recall also the general relation $(d_1 +1)\, \ddot{u}(0) = C + \epsilon u$, which follows
from the conservation law and the smoothness conditions at $t=0$. In generating the numerics,
we find it convenient to eliminate homothetic solutions by choosing $\epsilon$ to be $1$.
Furthermore, rather than setting $u(0)=0$, as was done throughout \S 1, we now set the constant $C$
to be zero. It then follows from the necessary condition ${\mathcal E} < 0$ that in the
series solution we must arrange for $\ddot{u}(0)= \frac{u(0)}{d_1+1} < 0$, with
$u(0)$ as an otherwise arbitrary parameter.

\medskip

\noindent{\bf Example 1.} \: We set $G={\rm Sp}(m+1),  H={\rm Sp}(m) \times {\rm Sp}(1),$
and $K={\rm Sp}(m) \times {\rm U}(1)$. The principal orbit $G/K$ is diffeomorphic to $\C{\PP}^{2m+1}$
and the singular orbit $G/H$ is $\HH{\PP}^m$. So $d_1=2, d_2 = 4m$, and $A_2 = 2m(m+2), A_3=\frac{m}{2}$
(with $\mathsf b$ chosen to be $-2\tr(XY)$). The initial values of $(z_1, \ldots, z_6)$ are given by
$(0, 1, \bar{h}, 0, \bar{u}, 0)$ where $\bar{h} > 0$ and $\bar{u} < 0$. These give rise to a
$2$-parameter family of numerical solutions.

In Figures 1 and 2 below we plot the functions $g_i$ and $u$ for the
$m=1$ and $m=2$ cases respectively with parameter values $\bar{h} = 6$ and $\bar{u} = -1$.

Note that the soliton potential is concave down and becomes approximately quadratic,
in accordance with Proposition \ref{expander-pot} and Proposition
\ref{expander-gradlb}. The $g_i$ are asymptotically linear.

\pagebreak

\begin{figure}[hbtp]
\includegraphics[width=15cm,height=9cm]{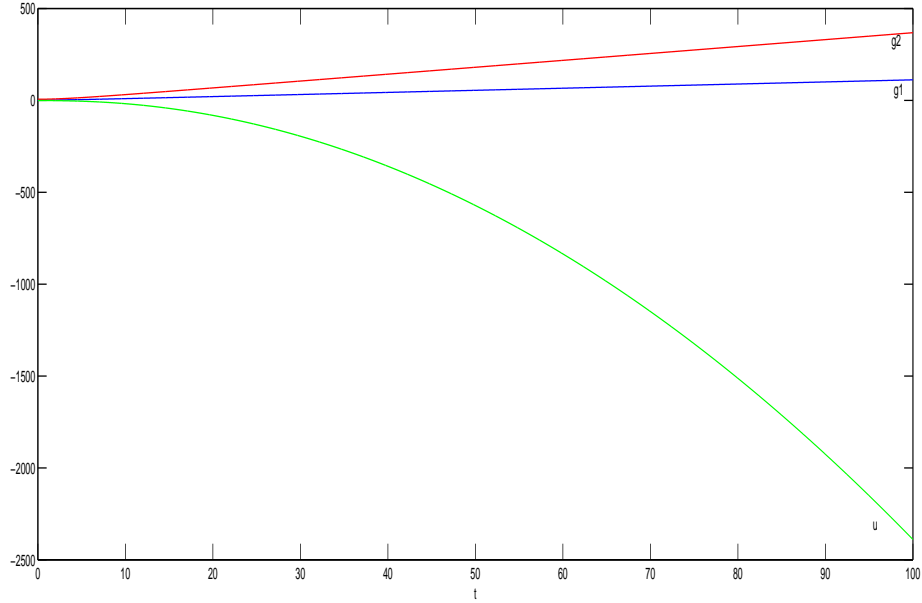}
\caption{$m=1$ case}
\end{figure}

\begin{figure}[hbtp]
\includegraphics[width=15cm,height=9cm]{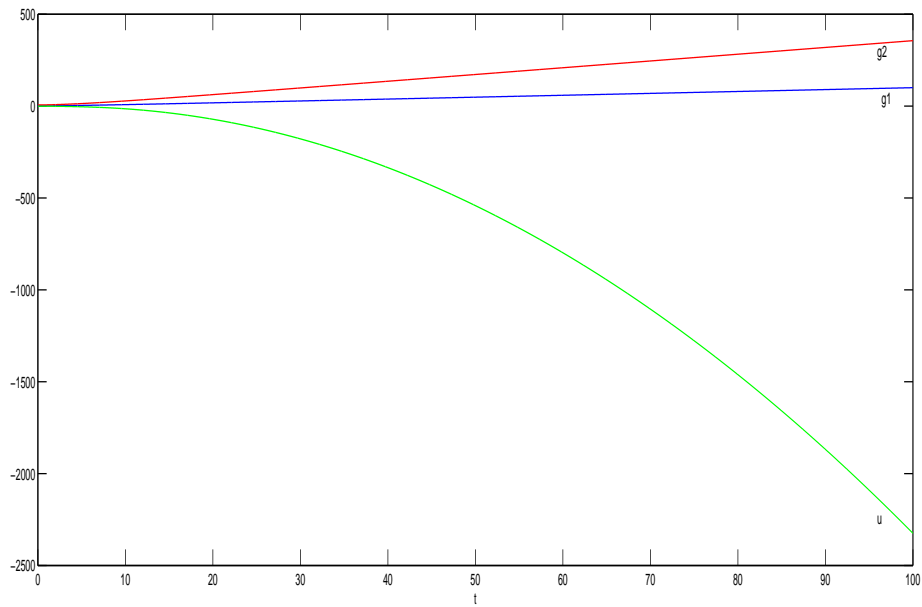}
\caption{$m=2$ case}
\end{figure}

We have also plotted the quantities ${\tilde X}_i =  \frac{X_i}{\sqrt{d_i}}$ and
${\tilde Y}_i = \frac{Y_i}{\sqrt{d_1}}$ against $t$ in Figures 3 and 4
for the $m=1$ and $m=2$ cases respectively. They all converge quickly to $0$.

\begin{figure}[hbtp]
\includegraphics[width=15cm,height=9cm]{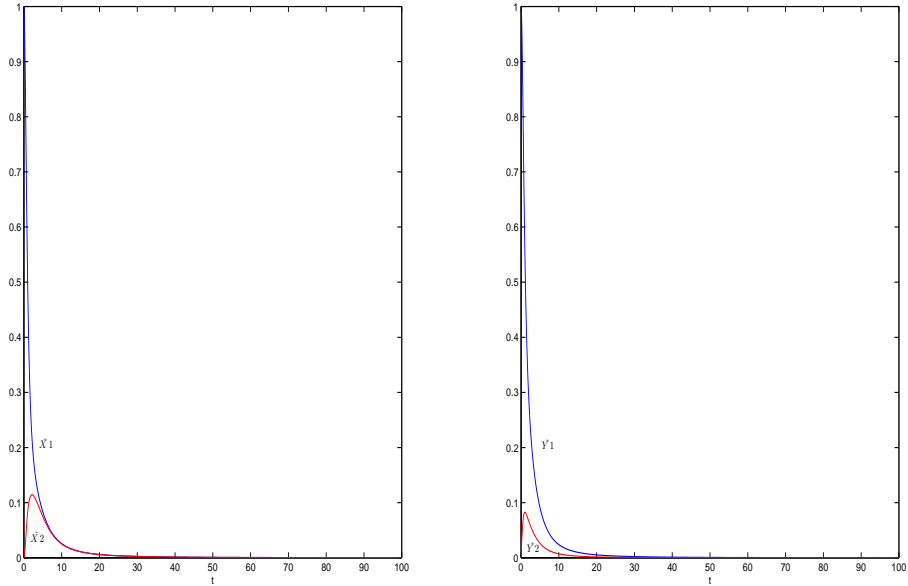}
\caption{$m=1$ case}
\end{figure}

\begin{figure}[hbtp]
\includegraphics[width=15cm,height=9cm]{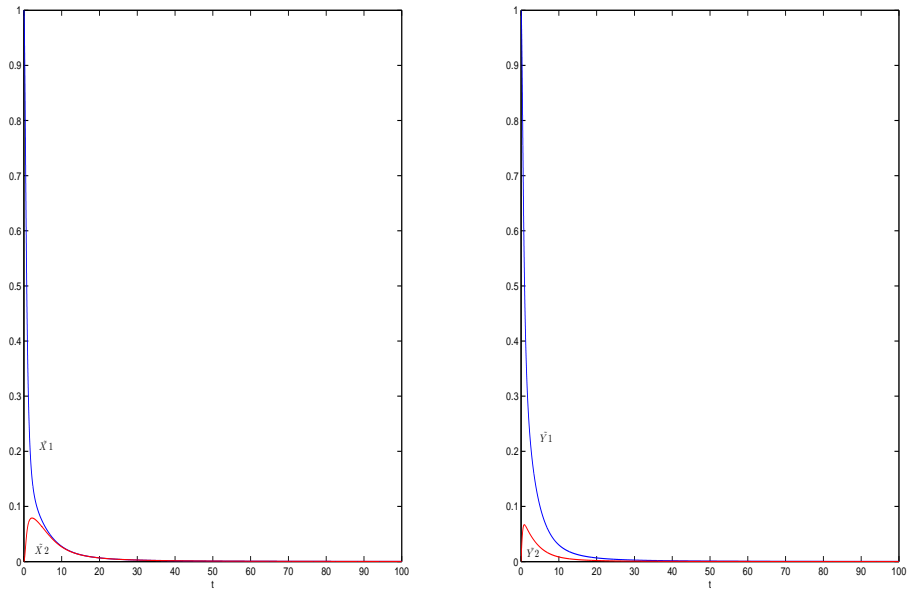}
\caption{$m=2$ case}
\end{figure}

In Figure 5 we plot the ratios ${\tilde X}_1/{\tilde X}_2$ and ${\tilde Y}_1/{\tilde Y_2}$.
Note that the second ratio is $\frac{g_2}{g_1}$, which tends to a positive constant. The first ratio
is the ratio of the principal curvatures $\frac{\dot{g}_2}{g_2}/ \frac{\dot{g}_1}{g_1}$ and
we see that it quickly approaches $1$.

\begin{figure}[hbtp]
\includegraphics[width=17cm,height=9cm]{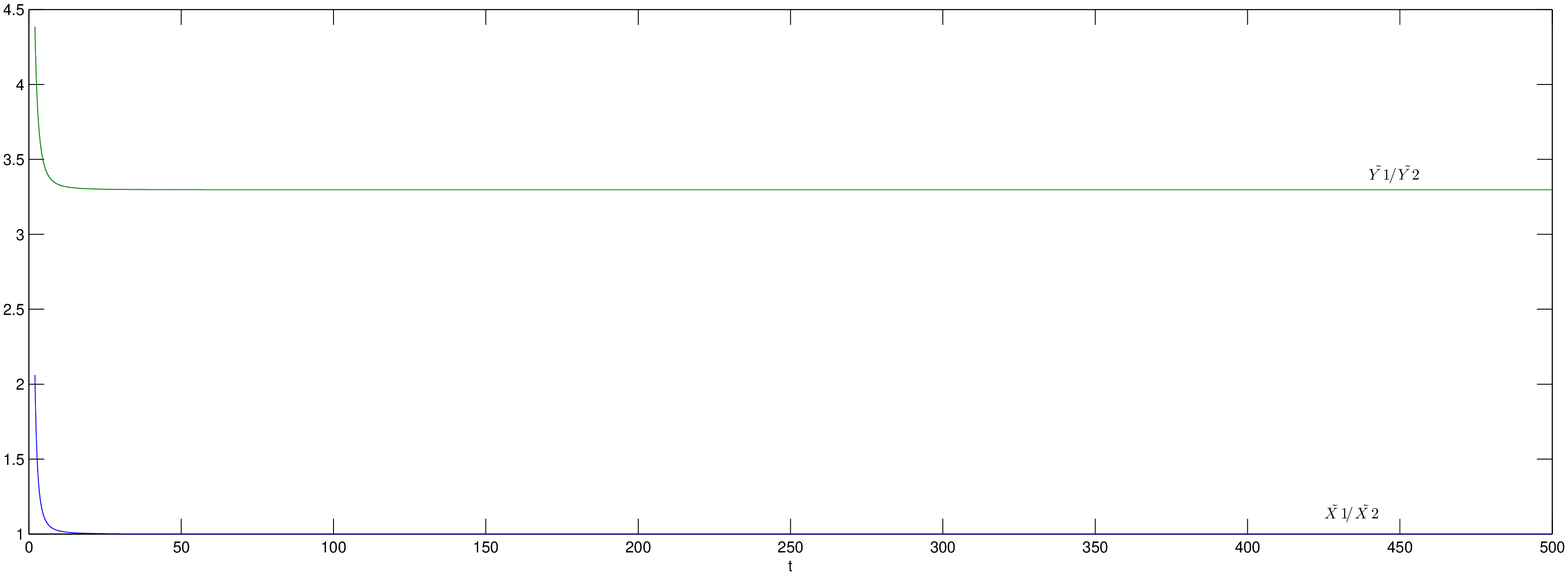}
\caption{}
\end{figure}

Similar numerical results hold for larger values of $m$.

\pagebreak

\noindent{\bf Example 2.} \: We next set  $G={\rm Sp}(m+1) \times {\rm Sp}(1),
H={\rm Sp}(m) \times {\rm Sp}(1) \times {\rm Sp}(1),$ and $K={\rm Sp}(m) \times \Delta{\rm Sp}(1)$.
The principal orbit $G/K$ is diffeomorphic to $S^{4m+3}$ and the singular orbit $G/H$ is
again $\HH{\PP}^m$. So $d_1=3, d_2 = 4m$, and $A_2 = 4m(m+2), A_3=\frac{3m}{4}$
(where $\mathsf b$ is given by $-2\tr(XY)$ on both of the simple factors).
The initial values of $(z_1, \ldots, z_6)$ are given by
$(0, 1, \bar{h}, 0, \bar{u}, 0)$ where $\bar{h} > 0$ and $\bar{u} < 0$.

For this case we obtain graphs very similar to those in Example 1.

\medskip

Based on the last two examples, we would conjecture that on the
vector bundles $G \times_H \R^{d_1+1},$ where $(G, H,  K)$ are as above, there is a
$2$-parameter family of non-homothetic complete expanding  gradient Ricci solitons.


\begin{thebibliography}{bbbbbb}


\bibitem[BB]{BB} L. B\'erard Bergery, {\em Sur des nouvelles vari\'et\'es
    Riemanniennes d'Einstein}, Publications de l'Institut Elie Cartan,
     Nancy, (1982).
\bibitem[Bo]{Bo} C. B\"ohm, {\em Non-compact cohomogeneity one Einstein manifolds,}
      Bull. Soc. Math. France, {\bf 122}, (1999), 135-177.
\bibitem[BGK1]{BGK1} C. Boyer, K. Galicki and J. Koll{\' a}r, {\em Einstein metrics on
       spheres}, Ann. Math., {\bf 162}, (2005), 557-580.
\bibitem[BGK2]{BGK2} C. Boyer, K. Galicki and J. Koll{\' a}r, {Einstein metrics on
    exotic spheres in dimension $7, 11, $ and $15$,} {\em Experiment. Math.,} {\bf 14},
    (2005), 59-64.
\bibitem[Bry]{Bry} R. Bryant, unpublished work.
\bibitem[Buz]{Buz} M. Buzano, {\em Initial value problem for cohomogeneity one
         gradient Ricci solitons}, J. Geom. Phys, {\bf 61}, (2011), 1033-44.
\bibitem[BDGW]{BDGW} M. Buzano, A. Dancer, M. Gallaugher and M. Wang,
            {\em A family of steady Ricci solitons and Ricci-flat metrics},
             arXiv:math.DG//1309.6140.
\bibitem[Cao]{Cao} H. D. Cao, {\em Existence of gradient Ricci solitons},
     Elliptic and Parabolic Methods in Geometry, A. K. Peters,
     (1996), 1-16.
\bibitem[CaNi]{CaNi} J. A. Carrillo and L. Ni, {\em Sharp logarithmic Sobolev
     inequalities on gradient solitons and applications,} Comm. Anal. Geom.,
      {\bf 17}, (2009), 721-753.
\bibitem[Chb]{Chb} B. L. Chen, {\em Strong uniqueness of the Ricci flow},
       J. Diff. Geom., {\bf 82 }, (2009),363-382.
\bibitem[Chc]{Chc}  Chih-Wei Chen, {\em On the asymptotic behavior of
    expanding gradient Ricci solitons,}, Ann. Glob. Anal. Geom., {\bf 42}, (2012), 267-277.
\bibitem[Cho]{Cho} O. Chodosh, {\em Expanding Ricci solitons asymptotic to cones,}
        arxiv:math.DG//1303.2983.
\bibitem[DHW]{DHW} A. Dancer, S. Hall and M. Wang, {\em Cohomogeneity one
          shrinking Ricci solitons: an analytic and numerical study},
           Asian J. Math., {\bf 17}, (2013), no. 1, 33-61.
\bibitem[DW1]{DW1} A. Dancer and M. Wang, {\em On Ricci solitons
             of cohomogeneity one}, Ann.  Glob. Anal. Geom., {\bf 39}, (2011) 259-292.
\bibitem[DW2]{DW2} A. Dancer and M. Wang, {\em Some new examples of
      non-K\"ahler Ricci solitons}, Math. Res. Lett. {\bf 16}, (2009) 349-363.
\bibitem[DW3]{DW3} A. Dancer and M. Wang, {\em Non-K\"ahler expanding
           Ricci solitons}, IMRN, ({\bf 2009}), 1107-33.
\bibitem[DW4]{DW4}  A. Dancer and M. Wang, {\em The cohomogeneity one Einstein equations
           from the Hamiltonian viewpoint},  \ J. reine angew. Math., {\bf 524}, (2000), 97-128.
\bibitem[FR]{FR} M. Fern{\'a}ndez-L{\'o}pez and E. Garc{\'i}a-R{\'i}o, {\em Maximum
      principles and gradient Ricci solitons,} J. Diff. Equations, {\bf 251}, (2011), 73-81.
\bibitem[Gaf]{Gaf} M. P. Gaffney, {\em A special Stoke's theorem for complete
      Riemannian manifolds,} Ann. Math., {\bf 60}, (1954), 140-145.
\bibitem[GK]{GK} A. Gastel and M. Kronz, {\em A family of expanding Ricci solitons},
       Variational Problems in Riemannian Geometry, Prog. Nonlinear
       Differential Equations Appl. {\bf 59}, Birkh\"auser, Basel (2004), 81--93.
\bibitem[GHMS]{GHMS} M. Gutperle, M. Headrick, S. Minwalla and
     V. Schomerus, {\em Space-time energy decreases under world-sheet RG flow},
      JHEP, {\bf 1} (2003).
\bibitem[Iv]{Iv} T. Ivey, {\em New examples of complete Ricci solitons},
       Proc. AMS, {\bf 122}, (1994), 241-245.
\bibitem[Ja]{Ja} M. Jablonski, {\em Homogeneous Ricci solitons,} arXiv:math.DG//1109.6556.
\bibitem[Ka]{Ka} K. Kawakubo, {\em Smooth structures on $S^p \times S^q$,} Osaka
      J. Math., {\bf 6}, (1969), 165-196.
\bibitem[KwS]{KwS} S. Kwasik and R. Schultz, {\em Multiplication stablization and
      transformation groups,} in Current Trends in Transformation Groups,
      K-Monogr. Math., Kluwer, (2002), 147-165.
\bibitem[La1]{La1} J. Lauret, {\em Ricci soliton homogeneous nilmanifolds},
     Math. Ann., {\bf 319} (2001), 715-733.
\bibitem[La2]{La2} J. Lauret, {\em Einstein solvmanifolds and nilsolitons},
    Contemporary Mathematics, {\bf 491}, (2009), 1-35.
\bibitem[La3]{La3} J. Lauret, {\em Ricci soliton solvmanifolds,}
        J. reine angew. Math., {\bf 650}, (2011), 1-21.
\bibitem[MS]{MS} O. Munteanu and N. Sesum, {\em On gradient Ricci solitons},
      J. Geom. Anal. {\bf 23}, (2013), no. 2, 539-561.
\bibitem[PW]{PW} P. Petersen and W. Wylie, {\em Rigidity of gradient Ricci solitons,}
     Pacific J. Math., {\bf 241}, (2009), 329-345.
\bibitem[PrW]{PrW} M. Protter and H. Weinberger, {\em Maximum principles in differential
    equations}, Springer-Verlag, (1984).
\bibitem[PRS]{PRS} S. Pigola, M. Rimoldi and A. Setti, {\em Remarks on non-compact gradient
     Ricci solitons}, Math. Z., {\bf 268}, (2011), 777-790.
\bibitem[Sc]{Sc} R. Schultz, {\em Smooth structures on $S^p \times S^q$}, Ann. Math., {\bf 90},
       (1969), 187-198.
\bibitem[SS]{SS} F. Schulze and M. Simon, {\em Expanding solitons with non-negative curvature
    operator coming out of cones,} Math. Z., {\bf 275}, (2013), no. 1-2, 625-639.
\bibitem[Y]{Y} S. T. Yau, {\em Some function-theoretic properties of complete
        Riemannian manifold and their applications to geometry}, Indiana U. Math. J.,
        {\bf 25}, (1976), 659-670.
\bibitem[Zh]{Zh} Z.-H. Zhang, {\em On the completeness of gradient Ricci solitons},
       Proc. A. M. S., {\bf 137}, (2009), 2755-2759.
\end{thebibliography}
\end{document}